\documentclass{amsart}
\usepackage{amssymb,latexsym,amsmath,amsfonts}
\usepackage[active]{srcltx}

\newtheorem{lem}{Lemma}[section]
\newtheorem{prop}[lem]{Proposition}
\newtheorem{thm}[lem]{Theorem}

\theoremstyle{definition}
\newtheorem{df}[lem]{Definition}

\newcommand{\DD}{{\mathfrak D}}
\def\z{{\sf{z}}}
\def\des{{\sf Dev}}

\def\cS{\mathcal{S}}
\def\cD{\mathcal{D}}

\def\C{\mathbb{C}}
\def\F{\mathbb{F}}
\def\O{\mathcal{O}}

\def\bE{{\bf J}}
\def\normaleq{\trianglelefteq}

\def\ovr{\overline}

\def\Z{\mathbb{Z}}
\def\sig{\sigma}

\def\veps{\varepsilon}

\DeclareMathOperator{\aut}{\sf Aut}

\DeclareMathOperator{\fix}{\sf Fix}
\DeclareMathOperator{\sym}{\sf Sym}

\newcommand{\sg}[1]{\langle {#1}\rangle}

\newcommand{\comment}[1]{}

\newcommand{\e}[1]{#1}
\def\bC{{\bf C}}
\def\bN{{\bf N}}
\def\irr{{\sf Irr}}
\newcommand{\und}[1]{{\underline #1}}
\newcommand{\New}[1]{{\it #1}}

\begin{document}
\title{On Skew Hadamard difference sets}
\author{Mikhail Muzychuk}
\address{Netanya Academic College, Netanya, Israel}
\email{muzy@netanya.ac.il}
\date{}
\maketitle
\begin{abstract} In this paper we construct exponentionally many non-isomorphic skew Hadamard difference sets over an elementary abelian group of order $q^3$.
\end{abstract}

\section{Introduction}

During last 5 years there was ongoing activity related to skew Hadamard difference sets (all necessary
definitions are given in the next section). Apart the classical family of such sets there are only few infinite
series built quite recently. We refer the reader to the papers \cite{xiang},\cite{dingyuan},\cite{feng}
where the topic is surveyed.

The initial point for this paper was Feng's construction of skew Hadamard difference sets \cite{feng}
over non-abelian groups of order $p^3$ (see the last section where the history of this paper is presented).
An analysis of Feng's construction shows that the automorphism group of the design
obtained from his difference set contains an elementary abelian regular subgroup of order $p^3$.
This means that this design could be constructed from a difference set over an
elementary abelian group of order $p^3$.
This gives a clue how to generalize the construction to elementary abelian groups of order
$q^3$. We take a finite field $\F_q,q\equiv\,3({\rm mod}\,4)$ and consider
the orbits of a certain subgroup $A\leq {\sf GL}_3(\F_q)$ (see Section 3 for definition of $A$)
on the vector space $\F_q^3$.
The first main result of the paper (Theorem \ref{SHfusions}) enumerates all $A$-invariant skew Hadamard difference sets. It turns out that the number of such sets is $4{q\choose\frac{q+1}{2}}$. The second main result (Theorem~\ref{main2}) of the paper solves isomorphism problem for designs generated by the constructed difference sets. It turns out that the number of the design isomorphism classes
is at least $2^{q+1}/q^4$.

Each design obtained in that way is invariant under the action of the group $V\rtimes A$
of order $q^4\frac{q-1}{2}$. If $p > 3$, then the group $V\rtimes A$ contains a non-abelian
subgroup isomorphic to $UT_3(q)$ acting regularly on the point set. Therefore in the case of
$p > 3$ each design may be also derived from a skew Hadamard difference set over the non-abelian
group $UT_3(q)$.

 \section{Preliminaries}

Let $H$ be a finite group written multiplicatively with unit element $1_H$.
We write $\Z[H]$ for the group algebra of $H$ over the integers.
Given an element $x = \sum_{h\in H} x_h h\in \Z[H]$, we write $x^{(-1)}$ for
the element $\sum_{h\in H} x_h h^{-1}$.
For a subset $D\subseteq H$ we set $D^{(-1)}:=\{d^{-1}\,|\,d\in D\}$ and
$\und{D}:=\sum_{d\in D} d\in\Z[H]$.

\subsection{Schur rings}
Given a partition $\cS$ of $H$, we denote by $\cS(h),h\in H$ a unique class of $\cS$ containing $h$.
\begin{df} A partition $\cS$ of $H$ is called a \New{Schur partition} iff it satisfies the following conditions
\item[(S1)] $\cS(1_H)=\{1_H\}$;
\item[(S2)] $\cS(h^{-1}) = \cS(h)^{(-1)}$ for each $h\in H$ ;
\item[(S3)] the free $\Z$-submodule $\Z[\und{\cS}]$ spanned by $\und{S},S\in\cS$ is a subalgebra of $\Z[H]$. This subalgebra is called a \New{Schur ring/algebra} spanned by $\cS$.
\end{df}
If $H$ is abelian, then each Schur partition $\cS$ defines a \New{dual} Schur partition $\cS^*$ over
the dual group $\irr(H)$.Two irreducible characters $\chi,\eta\in\irr(H)$ belong to the same class of $\cS^*$ if and only if
$\chi(\und{S})=\eta(\und{S})$ holds for each $S\in\cS$. A \New{character table} of $\cS$ describes all irreducible complex representations of $\C[\und{\cS}]$. It's rows  are parametrized by the sets of $\cS^*$ while the columns are parametrized by the sets of $\cS$. The entry of the table corresponding to a pair $R\in\cS^*$,$S\in\cS$ is equal to $\chi(\und{S})$ where $\chi\in R$.

If $\Phi$ is a group of automorphisms of $H$, then the orbits of $\Phi$ on $H$ always form a Schur partition.
If $H$ is abelian, then the dual Schur partition is formed by the orbits of a natural action of $\Phi$ on $\irr(H)$.

\subsection{Difference sets and Cayley designs} A subset $D$ is called a $(v,k,\lambda)$-\New{difference set} iff $v = |H|,k = |D|$ and the following equation is satisfied
\begin{equation}\label{difset}
\und{D}\,\und{D}^{(-1)} =  (k - \lambda) 1_H + \lambda \und{H}
\end{equation}
A difference set $D$ is called a \New{skew Hadamard difference set} (SHDS for short) if $H\setminus\{1_H\}=D\cup D^{(-1)}$ and $D\cap D^{(-1)}=\emptyset$.
It is not difficult to check that a subset $D\subseteq H$ is an SHDS iff the partition $\{1_H\},D,D^{(-1)}$ is a Schur partition of $H$. The parameters of an SHDS over $H$ are $(|H|,\frac{|H|-1}{2},\frac{|H|-3}{4})$.

Every $(v,k,\lambda)$ difference set $D\subset H$ produces a $2-(v,k,\lambda)$ symmetric design with point set $H$ and block set $\des(H,D) :=\{Dh\,|\,h\in H\}$. If the group $H$ is clear from the context, then we write just $\des(D)$ instead of $\des(H,D)$.

In general, the set $\des(D)$ could be built for any subset $D$ of $H$. In what follows we call the set $\des(D)$ \New{a Cayley design generated by} $D$. In the case when  $H$ is an abelian group written additively we call $\des(D)$ a \New{translation design}  generated by $D$.

 The \New{automorphism group} of a design $(H,\cD), \cD=\des(D)$, notation $\aut(\cD)$, consists of all permutations $g\in\sym(H)$ which permute the blocks $B\in\cD$. It always contains a subgroup $H_\ast:=\{h_\ast\,|\,h\in H\}$  where $h_\ast\in\sym(H)$ is a right translation \footnote{ In the case when $H$ is an abelian group written additively we write $h_+$ instead of $h_\ast$.} by $h\in H$ (that is $x^{h_\ast} = xh$). The group $H_\ast$ acts regularly on points and transitively on the blocks of the design
$(H,\cD)$. This property is characteristic for Cayley designs. More precisely, a simple design is isomorphic to a Cayley design $\des(H,D)$ iff the automorphism group of the design contains a subgroup isomorphic to $H$ which acts regularly on points and transitively on blocks. Below we collect some elementary properties of Cayley designs which will be used in the paper.

\begin{prop}\label{pointblock} Let $\cD:=\des(H,D)$ be a Cayley design s.t. $|\cD|=|H|$ and $G$ a subgroup of $\aut(\cD)$ which contains $H_*$.
Assume that a point stabilizer $G_p,p\in H$ fixes some block $B\in\cD$. Then
\begin{enumerate}
\item If $G_p$ fixes a point $q\in H$, then $G_p$ fixes also a block $Bp^{-1}q$;
\item If $G_p$ fixes a block $Bh,h\in H$, then $G_p$ also fixes a point $ph^{-1}$;
\item there exists $F\leq H$ such that
$$
\fix_H(G_p) = \{pf\,|\,f\in F\}, \fix_{\cD}(G_B) = \{Bf\,|\,f\in F\}.
$$
\end{enumerate}
\end{prop}
\begin{proof}
{\sc Part (1).} Let $g\in G_p$ be an arbitrary element. Then $(Bp^{-1}q)^g = Bh$ for some $h\in H$. Equivalently,
$B^{p_*^{-1}q_* g} = Bh$. Since $G_p H = H G_p$, there exist $g_1\in G_p$ and $r\in H$ such that
$p_*^{-1}q_* g = g_1 r_*$. This implies that $Bh = B^{g_1 r_*} = Br$. Also
$$
p^{p_*^{-1}q_* g} = p^{g_1 r_*}\implies q^g = pr \implies q = pr\implies r = p^{-1}q.
$$
Thus
$(Bp^{-1}q)^g = Bh = Br = Bp^{-1}q$.

{\sc Part (2).} It follows from $(Bh)^{G_p} = Bh$ that $h_* G_p h_*^{-1}\leq G_B$. By assumption
$G_p\leq G_B$. Since $|G_B| = |G|/|\cD| = |G|/|H|=|G_p|$, we obtain $G_p = G_B$. Thus
$h_* G_p h_*^{-1} = G_p$ implying
$$
(ph^{-1})^{G_p} = (ph^{-1})^{h_* G_p h_*^{-1}} = ph^{-1}.
$$

{\sc Part (3).} It is well-known that $\fix_H(G_p) = p^{{\bf N}_G(G_p)}$. It follows from
$G = G_p H_*$ that ${\bf N}_G(G_p) = G_p F_*$ for a uniquely determined subgroup $F\leq H$.
Now we obtain that
$$
\fix_H(G_p) = p^{{\bf N}_G(G_p)} = p^{G_p F_*} = pF.
$$
The second equality follows from $G_B=G_p$.\qed
\end{proof}

\section{The group $A$ and its orbits}
\newcommand{\clmn}[3]
{\left(
\begin{array}{c}
#1\\
#2\\
#3
\end{array}
\right)
}

For the rest of the paper it is assumed that $q=p^n$ is an odd power of a prime $p$ which is congruent $3$ modulo $4$.
Let $V = \F_q^3$. We write the elements of $V$ as column vectors $\clmn{x}{y}{z}$ and the elements of the dual space
$V^*$ are written as row vectors. For $v=(v_1,v_2,v_3)\in V^*, w=\clmn{w_1}{w_2}{w_3}\in V$ their product $v_1w_1 + v_2w_2 + v_3w_3$ is written as $vw$. For $w=\clmn{w_1}{w_2}{w_3}$ we set $w^*:=(w_3,w_2,w_1)$.
Analogously $(v_1,v_2,v_3)^*:=\clmn{v_3}{v_2}{v_1}$. Notice that $vw = w^*v^*$.

For each $x\in\F_q$ we set
$$
E(x):=\left(
\begin{array}{ccc}
1 & x & x^2/2\\
0 & 1 & x \\
0 & 0 & 1
\end{array}
\right).
$$
The set $E:=\{E(x)\,|\,x\in\F_q\}$
is an  elementary abelian subgroup of $GL_3(\F_q)$ isomorphic to $(\F_q,+)$.
Let also $S:=\{sI_3\,|\,s\in\F_q^{*2}\}$.
The group
$A:=ES$ is an abelian group of odd order $q(q-1)/2$. In what follows  $V$ is considered as a left $A$-space
while $V^*$ is considered as a right $A$-space.  The following formula is straightforward
\begin{equation}\label{*}
\forall_{v\in V}\ \forall_{g\in A} \ (gv)^*=v^*g.
\end{equation}

To describe the orbits of $A$ on $V$ and $V^*$ we introduce a certain set which will be used as an index set
for $A$-orbits. Define $I:=\F_q\cup\{\infty,\bullet\}$ and for each $i\in I$ define the $A$-orbits $O_i,i\in I$ and $O^*_i,i\in I$ as follows. For $i\in\F_q$ we set
$$
O_i = A\clmn{i}{0}{1} =
\left\{\clmn{s\left(\frac{x^2}{2}+i\right)}{sx}{s}\,\vline\,x\in\F_q,s\in\F_q^{*2}\right\}, O^*_i:=(O_i)^*
$$
$$
O_\infty:= A\clmn{0}{1}{0} =
\left\{\clmn{x}{s}{0}\,\vline\,x\in\F_q,s\in\F_q^{*2}\right\}, O^*_\infty:=(O_\infty)^*
$$
$$
O_\bullet:=A\clmn{1}{0}{0} =
\left\{\clmn{s}{0}{0}\,\vline\,s\in\F_q^{*2}\right\}, O^*_\bullet:=(O_\bullet)^*
$$

\begin{prop}\label{orbits}
The sets $\{O_i,-O_i\,|\,i\in I\}$ and $\{O^*_i,-O^*_i\,|\,i\in I\}$ form complete sets of
 non-zero orbits of $A$ on $V$ and $V^*$, respectively.
\end{prop}
In what follows we set $\O:=\{O_i\,|\,i\in I\},-\O:=\{-O_i\,|\,i\in I\}$ and analogously for duals.

The orbits of $A$ on $V$ form a Schur partition $\cS$ while the $A$-orbits on $V^*$ form a dual Schur partition $\cS^*$.
A linear map ${}^*$ between $V$ and $V^*$ yields an isomorphism between these S-rings.
Notice that $\cS:=\O\cup -\O \cup \{\{0\}\}$ and $\cS^*:=\O^*\cup -\O^* \cup \{\{0\}\}$.

\section{The character table of $\cS$.}

Let  $\omega_p\in\C$ be a $p$-th primitive root of unity and $\tau:\F_q\rightarrow \C$ be the map defined by $\tau(x):=\omega_p^{tr(x)}$ where $tr(x)$ is an $\F_p$-trace of $x\in\F_q$. Clearly that $\tau(x+y)=\tau(x)\tau(y)$ for $x,y\in\F_q$.

The additive characters of $V$ are parametrized by the vectors of $V^*$ and the value $\chi_{v}(w)$ of the character  $\chi_v$
 (where $v\in V^*,w\in V$) is equal to $\tau(vw)$.

For the rest of the text we set $\z:=\sum_{x\in\F_q^{*2}} \tau(x)$ and $\Delta:=\z - \overline{z}$. It is well-known that
$$
\z=\frac{-1\pm\sqrt{-q}}{2}, \z + \overline{\z}=-1, \Delta = \pm \iota\sqrt{q}
$$
(here and later on $\ovr{z }$ means a complex conjugate of $z$).

The principal part of the character table of $\cS$ is a square matrix of size $2(q+2)$, its columns are parametrized by the non-zero $A$-orbits on $V$ while the rows are parametrized by the non-zero $A$-orbits on $V^*$. The value of the character table, denoted as $[R,O]$, corresponding to a pair of $A$-orbits $O\subset V,R\subset V^*$ is computed by the formula
$$
[R,O] = \sum_{w\in O}\chi_v(w).
$$
where $v\in R$ is an arbitrary element.

It follows from
$$
[R,-O]=[R,-O]=\overline{[R,O]}\mbox{ and } [-R,-O]=[R,O].
$$
that it is sufficient to compute only the numbers $[O^*_i,O_j]$ for $i,j\in I$.

In what follows we write $\sig(x),x\in\F_q^*$ for an automorphism of $\C$ which is identical if $x$ is a square and complex conjugation if $x$ is a non-square. Notice that for $a\in\F_q^*$ we always have that $\sum_{s\in\F_q^{*2}}\tau(as) = \z^{\sig(a)}$.

\begin{prop}\label{table} The values of $[O^*_i, O_j]$ are given in the following table
$$
\begin{array}{|c|c|c|c|}
\hline
\ & j\in \F_q & j = \infty & j = \bullet\\
\hline
i\in\F_q & P(i,j) & 0 & \z\\
\hline
i=\infty & 0 & q\z & \frac{q-1}{2}\\
\hline
i=\bullet & q\z & q\frac{q-1}{2} & \frac{q-1}{2}\\
\hline
\end{array}\\
$$
\centerline{Table 1}
where
\begin{equation}\label{entry1}
P(i,j)=\left\{
\begin{array}{rl}
\frac{q-1}{2}(1+2\z^{\sig(2)}) & i+j=0;\\
\z^{\sig(i+j)} (1+2\z^{\sig(2)}) & i+j\neq 0
\end{array}
\right. .
\end{equation}
\end{prop}
\begin{proof}
First we show that
$$
\frac{[O^*_i,O_j]}{|O_j|} = \frac{[O^*_j,O_i]}{|O_i|}.
$$
Pick arbitrary $v\in O_i,w\in O_j$. Then
$$
[O^*_i,O_j]=\sum_{u\in O_j}\chi_{v^*}(u) = \frac{|O_j|}{|G|} \sum_{g\in G}\chi_{v^*}(gw) =
\frac{|O_j|}{|G|} \sum_{g\in G}\tau(v^*(gw)).
$$
It follows from~\eqref{*} that
$$
v^*(gw) = (v^*g)w=w^*(v^*g)^*=w^*(g v).
$$
Therefore
$$
[O^*_i,O_j]=\frac{|O_j|}{|G|} \sum_{g\in G} \tau(w^*(g v)) =
\frac{|O_j|}{|G|} \sum_{g\in G} \chi_{w^*}(g v) = \frac{|O_j|}{|G|}\cdot\frac{|G|}{|O_i|}[O_j^*,O_i]
= \frac{|O_j|}{|O_i|}[O_j^*,O_i].
$$

Thus it is sufficient to check Table 1 for 6  cases only:
$$
i,j\in\F_q;\ i=\infty, j\in\F_q;\ i=\bullet,j\in F_q;\ i=j=\infty;\ i=\bullet,j = \infty;\  i=j=\bullet.
$$
Each of these cases is checked below separately.

\medskip

{\sc Case A.} $i,j\in\F_q$. Since the value of $[O^*_i,O_j]$ does not depend on a choice of $v\in O^*_i$ we can take $v = (1,0,i)$. Then
$$
[O^*_i,O_j] = \sum_{w\in O_j}\chi_v(w) = \sum_{w\in O_j} \tau(vw) =
\sum_{s\in\F_q^{*2},x\in\F_q} \tau\left(\left(1,0,i\right)\clmn{s\left(\frac{x^2}{2}+j\right)}{sx}{s}\right) =
$$
$$
\sum_{s\in\F_q^{*2},x\in\F_q} \tau\left(s\left(\frac{x^2}{2}+j\right) +si\right) =
\sum_{s\in\F_q^{*2},x\in\F_q} \tau\left(s\frac{x^2}{2}  + s(i+j)\right) =
$$
$$
\sum_{s\in\F_q^{*2}}\sum_{x\in\F_q} \tau\left(s\frac{x^2}{2}\right)\tau( s(i+j)) =
\sum_{s\in\F_q^{*2}}\tau( s(i+j)) \sum_{x\in\F_q} \tau\left(\frac{x^2}{2}\right) =
$$
$$
\sum_{s\in\F_q^{*2}}\tau( s(i+j))(1+2\z^{\sig(2)}).
$$
If $i+j=0$, then the latter sum equals to $(1+2\z^{\sig(2)}) \frac{q-1}{2}$.
If $i+j\neq 0$, then the latter sum is
$\z^{\sig(i+j)}(1+\z^{\sig(2)})$.

\medskip

{\sc Case B.} $i=\infty, j\in\F_q$.\\
For $v=(0,1,0)\in O^*_\infty$ we obtain
$$
[O^*_\infty,O_j] = \sum_{w\in O_j}\chi_{(0,1,0)}(w)
\sum_{s\in\F_q^{*2},x\in\F_q} \tau\left(\left(0,1,0\right)\clmn{s\left(\frac{x^2}{2}+j\right)}{sx}{s}\right) =
\sum_{s\in\F_q^{*2},x\in\F_q} \tau(sx) = 0.
$$

\medskip

{\sc Case C.} $i=\bullet, j\in\F_q$.\\
For $v=(0,0,1)\in O^*_\bullet$ we obtain
$$
[O^*_\bullet,O_j] =  \sum_{w\in O_j}\chi_{(0,0,1)}(w)  =
\sum_{s\in\F_q^{*2},x\in\F_q} \tau\left(\left(0,0,1\right)\clmn{s\left(\frac{x^2}{2}+j\right)}{sx}{s}\right) =
\sum_{s\in\F_q^{*2},x\in\F_q} \tau(s) =q \z.
$$

\medskip

{\sc Case D.} $i=j=\infty$.\\
Take $(0,1,0)\in O^*_\infty$. Then
$$
[O^*_\infty,O_\infty] =
 \sum_{w\in O_\infty}\chi_{(0,1,0)}(w) =
\sum_{s\in\F_q^{*2},x\in\F_q} \tau\left(\left(0,1,0\right)\clmn{x}{s}{0}\right) =
\sum_{s\in\F_q^{*2},x\in\F_q} \tau(s) = q\z.
$$

\medskip

{\sc Case E.} $i=\bullet, j=\infty$.\\
Take $(0,0,1)\in O^*_\bullet$. Then
$$
[O^*_\bullet,O_\infty] =
 \sum_{w\in O_\infty}\chi_{(0,0,1)}(w) =
\sum_{s\in\F_q^{*2},x\in\F_q} \tau\left(\left(0,0,1\right)\clmn{x}{s}{0}\right) =
\sum_{s\in\F_q^{*2},x\in\F_q} \tau(0) = q\frac{q-1}{2}.
$$

\medskip

{\sc Case F.} $i=j=\bullet$.\\
Take $(0,0,1)\in O^*_\bullet$. Then
$$
[O^*_\bullet,O_\bullet] =
 \sum_{w\in O_\bullet}\chi_{(0,0,1)}(w) =
\sum_{s\in\F_q^{*2}} \tau\left(\left(0,0,1\right)\clmn{s}{0}{0}\right) =
\sum_{s\in\F_q^{*2}} \tau(0) = \frac{q-1}{2}.
$$
\qed
\end{proof}

\section{$A$-invariant skew Hadamard difference sets}

\begin{prop}\label{SH}
A subset $D\subseteq V\setminus\{0\}$ is a skew Hadamard difference set iff
it satisfies the following conditions:
\begin{itemize}
\item[{\rm (SH1)}] $- D\cap D=\emptyset$;
\item[(SH2)]  $-D\cup D = V\setminus\{0\}$;
\item[(SH3)]  $\chi_v(D)-\overline{\chi_v(D)} = \pm \iota q\sqrt{q}$ for each $v\in V^*\setminus\{0\}$.
\end{itemize}
\end{prop}
\begin{proof} If $D$ satisfies (SH1)-(SH2), then $\chi_v(D)+\overline{\chi_v(D)} = -1$ for
 each $v\in V^*\setminus\{0\}$, Together with (SH3) this implies that $\chi_v(D) =\frac{1\pm\sqrt{-q^3}}{2}$
for  each $v\in V^*\setminus\{0\}$. By \cite{wenghu}, Lemma 2.5 $D$ is a skew Hadamard difference set.

Vice versa, if $D$ is a skew Hadamard difference set, then (SH1)-(SH2) follow directly from the definition.
The condition (SH3) is a consequence of Lemma 2.5, \cite{wenghu}.\qed
\end{proof}
Let $\DD$ denote the set of all
 $A$-invariant  skew Hadamard difference sets, that is $D\in\DD$ iff $D$ is a union of $A$-orbits. It follows from (SH1)-(SH2) that a subset $D\in\DD$ contains exactly one orbit from $\{O_i,-O_i\}$ for each $i\in I$. In other words, the intersection $D\cap (O_i\cup -O_i)$ is either $O_i$ or $-O_i$. Therefore we can write $D\cap (O_i\cup -O_i) = \veps(i) O_i$ where $\veps(i)=\pm 1$.
Thus any $A$-invariant skew Hadamard difference set is uniquely determined by a function $\veps:I\rightarrow\{\pm 1\}$. In what follows we denote $D_\veps:=\bigcup_{i\in I} \veps(i)O_i$.

Although there are $2^{q+2}$ functions from $I$ into $\{\pm 1\}$, not every one of them yields a SHDS. A subset $D_\veps$ will be a SHDS iff it satisfies condition (SH3). Notice that
$$
D_{-\veps} = -D_{\veps}, D_{-\veps} \cap D_{\veps} = \emptyset, D_{-\veps} \cup D_{\veps} = V\setminus\{0\}.
$$
Thus $D_{\veps}$ is a SHDS iff $D_{-\veps}$ is a SHDS (the complementary  SHDS).

In order to describe all $A$-invariant SHDS we define $J_\veps:=\{i\in\F_q\,|\,\veps(i)=1\}$.
\begin{thm}\label{SHfusions}Let $\veps:I\rightarrow \{\pm 1\}$ be an arbitrary function.
Then $D_\veps\in\DD$
 iff $|J_\veps|=\frac{q+\mu}{2}$ and $\veps(\bullet) =\left(\frac{2}{p}\right)\mu$ where $\mu = \pm 1$.
\end{thm}
\begin{proof}
Pick an arbitrary vector $v\in V^*\setminus\{0\}$. Then
$$
\chi_v(D_\veps) - \ovr{\chi_v(D_\veps)} = \sum_{i\in I} \chi_v(\veps(i)O_i)  -\ovr{\chi_v(\veps(i)O_i)}
$$
Taking into account that  $\chi_v(-T) = \ovr{\chi_v(T)}$ we may rewrite the above equality as follows
\begin{equation}\label{eq1}
\chi_v(D_\veps) - \ovr{\chi_v(D_\veps)} =\sum_{i\in I} \veps(i)(\chi_v(O_i)  -\ovr{\chi_v(O_i)})
\end{equation}
Thus a subset $D_\veps$ will satisfy (SH3) (and, therefore, will be a SHDS) iff the above expression will be  equal to $\pm\iota q\sqrt{q}$ for
all $v\in V^*\setminus\{0\}$.
The value of the above sum depend only on the orbit to which the vector $v$ belongs. Moreover it follows from $\chi_{-v}(T) =\ovr{\chi_v(T)}$ that it is enough to check~\eqref{eq1} only for $v\in O^*_i,i\in I$.

Now we build a square matrix $T$ of order $q+2$ the rows and columns of which are indexed by the elements of $I$ and $T_{i,j}$ is the value $\chi_v(O_j) - \ovr{\chi_v(O_j)}$
where $v\in O^*_i$.

Using this matrix one can reformulate the condition for $\veps$ to produce a skew Hadamard difference set. Namely, the function $\veps$ produces a skew Hadamard difference set iff $T\veps^t$ is a column vector with entries $\pm \iota q\sqrt{q}$.

If $i,j\in\F_q$, then by~\eqref{entry1} we obtain that
\begin{equation}\label{T}
T_{i,j} =
\left\{
\begin{array}{rl}
(q-1)(\z_1-\ovr{\z_1}) & i+j=0;\\
\z_2 (1+2\z_1) - \ovr{\z_2(1+2\z_1) }  & i+j\neq 0
\end{array}
\right.
\end{equation}
where $\z_1 = \z^{\sig(2)}$ and $\z_2 = \z^{\sig(i+j)}$.

If $i+j=0$, then
$T_{i,j} =(q-1)\Delta^{\sig(2)}$.

Assume now that $i+j\neq 0$.

If $\sig(i+j)=\sig(2)$, then $\z_2 = \z_1$ and
$$
T_{i,j} = \z_1(1+2\z_1)-\ovr{\z_1(1+2\z_1)}=(\z_1 - \ovr{\z_1})(1 + 2(\z_1 + \ovr{\z_1})) =
\ovr{\z_1} - \z_1.
$$

If $\sig(i+j)\neq\sig(2)$, then $\z_2 = \ovr{\z_1}$ and
$$
T_{i,j} =\ovr{\z_1}(1+2\z_1)-\z_1(1+2\ovr{\z_1}) = \ovr{\z_1} - \z_1.
$$
Thus  $T_{i,j}=\ovr{\z_1}-\z_1 = - \Delta^{\sig(2)}$ whenever $i+j\neq 0$.

Using Table 1 one can finally compute $T$. It is more convenient to replace $T$ be the matrix $T'$ obtained from $T$ by row permuatation $i\mapsto -i,i\in \F_q$ ( the rows and $\infty,\bullet$ are not moved). The matrix $T'$ has the following block form
$$
T' = \left(
\begin{array}{c|c|c|c}
\ & \F_q & \infty & \bullet \\
\hline
\F_q & \Delta^{\sig(2)}(q{\bf I}_q - {\bf J}_q) & {\bf 0}^t & \Delta {\bf 1}^t \\
\hline
\infty & {\bf 0} & q\Delta & 0 \\
\hline
\bullet & q\Delta{\bf 1} & 0 & 0
\end{array}
\right)
$$
where ${\bf 0}$ and ${\bf 1}$ are zero and all-one row vectors of length $q$; and ${\bf I}_q$,${\bf J}_q$ are the identity and all-one matrices of size $q$. In order to compute the product $T'\veps^t$ we write $\veps^t$ in a block form
$$
\veps^t=\clmn{\veps_0^t}{\veps(\infty)}{\veps(\bullet)}
$$
where $\veps_0 :=\veps_{\F_q}$ is the restriction of $\veps$ onto $\F_q$.
In this notation $T'\veps^t$  has the following form
\begin{equation}\label{eq2}
\left(
\begin{array}{c|c|c}
\Delta^{\sig(2)} (q{\bf I}_q - {\bf J}_q) & {\bf 0}^t & \Delta {\bf 1}^t \\
\hline
{\bf 0} & q\Delta & 0 \\
\hline
q\Delta{\bf 1} & 0 & 0
\end{array}
\right)
 \clmn{\veps_0^t}{\veps(\infty)}{\veps(\bullet)} =
\clmn{q\Delta^{\sig(2)}\veps_0^t +(\Delta\veps(\bullet)- \Delta^{\sig(2)} \mu) {\bf 1}^t}{q\Delta\veps(\infty)}{q\Delta\mu},
\end{equation}
where $\mu$ is the coordinate sum of the vector $\veps_0$.
The entries of the right side of \eqref{eq2} are equal to $\pm\iota q\sqrt{q} = \pm q\Delta$ iff
$\mu = \pm 1$ and $\Delta\veps(\bullet) = \Delta^{\sigma(2)} \mu$. Together with $\Delta^{\sig(2)}=\left(\frac{2}{p}\right)\Delta$ we obtain that $\veps(\bullet) = \left(\frac{2}{p}\right)\mu$.
To finish the proof it remains to notice that $|J_\veps|=\frac{q+\mu}{2}$.
\qed
\end{proof}

An elementary counting shows that there are
$4{q\choose\frac{q+1}{2}}$ functions satisfying the conditions of Theorem~\ref{SHfusions}.
Therefore we obtain $4{q\choose\frac{q+1}{2}} > 2^{q+2}/q$ skew Hadamard difference sets (including complements). Notice that two distinct difference sets may be equivalent, and, therefore, produce isomorphic designs. For example, any element $g\in {\bf N}_{\aut(V)}(A)$  permutes the $A$-orbits, and therefore, also permutes the elements of  $\DD$ mapping each set $D\in\DD$ to an equivalent one.

\section{Isomorphisms between the translation designs}

The main result of this section solves the isomorphism problem for translation designs generated by difference sets from $\DD$.
\begin{thm}\label{main2}
Given $D,D'\in\DD$, the designs $\des(D)$ and $\des(D')$ are isomorphic iff they are isomorphic  by an element of ${\bf N}_{\aut(V)}(A)$.
\end{thm}

This Theorem gives us a lower bound for the number of non-isomorphic designs obtained from the difference sets constructed in the previous sections. The number of distinct designs is $4{q\choose\frac{q+1}{2}}$. Each design is $A$-invariant. Therefore each orbit of $\bN_{\aut(V)}(A)$ contains at most $[\bN_{\aut(V)}(A):A]$ designs. In the next subsection we'll show that $|\bN_{\aut(V)}(A)| = n(q-1)^2 q^2$. Therefore
$[\bN_{\aut(V)}(A):A] = 2n(q-1)q$ implying that the number of non-isomorphic designs is at least
$$
\frac{4{q\choose\frac{q+1}{2}}}{2n(q-1)q} > \frac{2^{q+2}/q}{2nq(q-1)} > \frac{2^{q+1}}{q^4}.
$$
\subsection{Computation of $\bN_{\aut(V)}(A)$}

In order to describe the group in the title
we'll write each automorphism of $V$ as $3\times 3$-matrix the entries of which are elements of the $\F_p$-algebra ${\sf End}(\F_q)$ (the algebra of $\F_p$-linear  endomorphisms of $\F_q$). The field $\F_q$ is considered as a subalgebra of ${\sf End}(\F_q)$ where the field element $\alpha$ is identified with  the endomorphism
$x\mapsto \alpha x,x\in\F_q$.

For each $\alpha,\beta\in\F_q^*$ and $\ell\in{\sf End}(\F_q)$ we set
$$
K(\alpha,\beta):=\left(
\begin{array}{ccc}
\e{\alpha} & 0 & 0\\
0 & \e{\alpha}\e{\beta} & 0\\
0 & 0 & \e{\alpha}\e{\beta}^2
\end{array}
\right),
\hat{\ell}:=\left(
\begin{array}{ccc}
\e{1} & 0 & \ell \\
0 & \e{1} & 0\\
0 & 0 & \e{1}
\end{array}
\right)
$$
Notice that $K:=\{K(\alpha,\beta)\,|\,\alpha,\beta\in\F_q^*\}$, $L:=\{\hat{\ell}\,|\,\ell\in{\sf End}(\F_q)\}$
are abelian subgroups of ${\sf Aut}(V)$. The first one is isomorphic to $\Z_{q-1}\times\Z_{q-1}$ while
the second one is an elementary abelian group isomorphic to $({\sf End}(\F_q),+)$. A direct check shows that both subgroups normalize $E$. Moreover $[E,L]=1$.

Another subgroup of ${\sf Aut}(V)$ normalizing $E$ comes from field automorphisms. More precisely  each $f\in {\sf Gal}(\F_q/\F_p)$ induces an $\F_p$-linear automorphism of $V$: $(x,y,z)\mapsto (f(x),f(y),f(z))$. We denote this automorphism by the same letter $f$. The subgroup of ${\sf Aut}(V)$ consisting of all Galois automorphisms $f\in {\sf Gal}(\F_q/\F_p)$ will be denoted as $F$. So, $F$ is a cyclic subgroup of ${\sf Aut}(V)$ of order $n$. A direct check shows that
$$
[F,E]\leq E, [K,E]\leq E, [L,E]=1, [F,K]\leq K,[F,L]\leq L,[K,L]\leq L
$$
In particular, any two subgroups from the list $\{E,F,K,L\}$ are permutable. Therefore $FKEL$ is a subgroup of
${\sf Aut}(V)$. It's order is equal to $|F||K||E||L| = n (q-1)^2 q^{n+1}$.

\begin{prop}\label{norm}  It holds that
\begin{enumerate}
\item ${\bN}_{\aut(V)}(E) = FKLE$;
\item  ${\bf N}_{\aut(V)}(SE)=FKEU$ where $U:=\{\hat{\gamma}\,|\,\gamma\in\F_q\}$;
\end{enumerate}
\end{prop}
\begin{proof}
{\sc Part (1).}
The subgroups ${\bf C}_V(E) = \{(x,0,0)^t|x\in \F_q\}, [E,V] = \{(x,y,0)^t|x,y\in \F_q\}$ are
${\bf N}_{\aut(V)}(E)$-invariant. Therefore each element $N\in {\bf N}_{\aut(V)}(E)$ has the following form
$$
 N =
\left(
\begin{array}{ccc}
a & b & c \\
0 & d & e \\
0 & 0 & f
\end{array}
\right)\mbox{ where } a,d,f\in {\sf GL}(\F_q)\mbox{ and } b,c,e\in {\sf End}(\F_q).
$$
Since $E\cong \F_q$, there exists an automorphism $n\in{\sf GL}(\F_q)$ such that
$$
\forall_{\alpha\in \F_q} \ N E(\alpha)  N^{-1} = E(n(\alpha)) \iff N E(\alpha) = E(n(\alpha)) N.
$$
After multiplication and equating we obtain the following equations which hold for each $\alpha\in\F_q$
\begin{equation}\label{eq3}
\left\{
\begin{array}{rcl}
a\alpha & = & n(\alpha) d,\\
d\alpha & = & n(\alpha) f,\\
\frac{1}{2} a\alpha^2 + b\alpha & = & \frac{1}{2} n(\alpha)^2 f + n(\alpha) e
\end{array}\right.\iff
\left\{
\begin{array}{rcl}
a\alpha & = & n(\alpha) d,\\
d\alpha& = & n(\alpha) f,\\
 b\alpha & = & n(\alpha) e
\end{array}\right.
\end{equation}
After substitution $\alpha =1$  we obtain $a = \beta d,d=\beta f,b=\beta e$
where $\beta:=n(1)$. Now the equations~\eqref{eq3} yield us
\begin{equation}\label{eq4}
\forall_{\alpha\in\F_q}
\left\{
\begin{array}{rcl}
d\alpha d^{-1} & = & \e{\beta}^{-1}n(\alpha),\\
f\alpha f^{-1} & = & \e{\beta}^{-1}n(\alpha),\\
 e\alpha e^{-1} & = & \e{\beta}^{-1}n(\alpha)
\end{array}\right.
\end{equation}
Thus each of the elements $e,d,f$ normalizes the subalgebra $\e{\F_q}$ of ${\sf End}(\F_q)$. Therefore each of them may be written in a form $e = \e{\epsilon}e_0,d=\e{\delta}d_0,f=\e{\phi}f_0$ where $e_0,d_0,f_0\in {\sf Gal}(\F_q/\F_p)$ and $\epsilon,\delta,\phi\in\F_q$. It follows from~\eqref{eq4} that conjugation by $a,d,f$ induce the same automorphism of $\e{\F_q}$. Therefore $e_0 = d_0 = f_0$. Thus we obtain the following
$$
a = \e{\beta} d = \e{\beta}^2 f = \e{\beta^2\phi} f_0;
d = \e{\beta} f = \e{\beta\phi} f_0;
e = \e{\epsilon} f_0, b = \e{\beta}e = \e{\beta\epsilon} f_0
$$
implying that
$$
N =
\left(
\begin{array}{ccc}
 \e{\beta^2\phi} f_0 &  \e{\beta\epsilon} f_0 & c \\
0 & \e{\beta\phi} f_0 &  \e{\beta\epsilon} f_0 \\
0 & 0 & \e{\phi} f_0
\end{array}
\right)\in FKEL
$$

{\sc Part (2).} Since $SE$ is a coprime product of $S$ and $E$, we conclude $\bN_{\aut(V)}(SE) = \bN_{\aut(V)}(S)\cap \bN_{\aut(V)}(E) = \bN_{FKEL}(S)$.
Since $F$, $K$ and $E$ normalize $S$, we can write  $\bN_{FKEL}(S)= (FKE)\bN_{L}(S)$. So, it remains to find $\bN_{L}(S)$. A straightforward  computation shows that  $\bN_{L}(S) = U$.

 \qed
\end{proof}

\subsection{An automorphism group of a concrete translation design}

Let us fix a function $\veps:I\rightarrow \{\pm 1\}$ such that $D := D_\veps\in\DD$. We also set
$ J:=J_\veps$, $\cD:=\des(V,D)$,
$G:=\aut(\cD)$ . We also denote by $F_p$ a unique Sylow $p$-subgroup of $F$.

Since $F$ normalizes $U$, it's Sylow $p$-subgroup $F_p$ normalizes $U$ as well. Therefore $F_pU = UF_p$ is a subgroup of $\aut(V)$.
Let  $Q:=(UF_p)_D$ be a setwise stabilizer of $D$ in $UF_p $. Since $UF_p\cong \F_q\rtimes F_p$, the group $Q$ is a subgroup of $\F_q\rtimes F_p$.
\begin{prop}\label{Q}
The subgroup $Q$ consists of all products $\hat{\alpha}f,\alpha\in\F_q,f\in F_p$ which satisfies the condition
$f(J) + \alpha = J$.
\end{prop}
\begin{proof} It follows from the definition of $U$ that for each $\hat{\alpha}\in U,\alpha\in\F_q$
and $i\in\F_q$ the following equalities hold
$$
\hat{\alpha}O_i = O_{i+\alpha},\hat{\alpha}O_\infty = O_\infty,\hat{\alpha}O_\bullet = O_\bullet.
$$
Also for each $f\in F$
$$
fO_i = O_{f(i)}, fO_\infty = O_\infty,fO_\bullet = O_\bullet.
$$
Hence
$$
\hat{\alpha}fD = \left(\bigcup_{i\in F_q} \veps(i)O_{f(i)+\alpha}\right) \bigcup \veps(\infty) O_\infty \bigcup \veps(\bullet) O_\bullet.
$$
Therefore $\hat{\alpha}fD = D$ holds if and only if $\veps(f(i)+ \alpha)=\veps(i)$ for each $i\in\F_q$. This is equivalent to $f(J) + \alpha = J$.\qed
\end{proof}

The subgroup $Q\cap U$ consists of those $\hat{\alpha},\alpha \in\F_q$ which satisfy $J+ \alpha= J$. Thus $J$ is a union of $\langle \alpha\rangle$-cosets. Since $|J|$ is coprime to $p$, we conclude that $Q\cap U$ is trivial. This implies that $Q$ is embedded into $F_p$. In particular, $Q$ is a cyclic $p$-group.

Since $F$ and $U$ normalize $EV_+$, the group $Q$ normalizes $EV_+$ too. Therefore $QEV_+$ is a subgroup of $G$.
\begin{thm}\label{sylow} $QEV_+$ is a Sylow $p$-subgroup of $G$.
\end{thm}

In order to prove this Theorem we first need some properties of $QEV_+$. Recall that the Thompson subgroup  $\bE(P)$ of a $p$-group $P$ is the subgroup generated by all elementary abelian subgroups of $P$ of maximal order.
Clearly $\bE(P)$ is characteristic in $P$. Also the subgroup $\bE_2(P)$ defined by $\bE_2(P)/\bE(P) = \bE(P/\bE(P))$ is characteristic in $P$.
\begin{prop}\label{regsubgroups}
Every elementary abelian subgroup of $EV_+$ of order $\geq q^2p$ is contained in $V_+$. In particular,  $\bE(EV_+)=V_+$.
\end{prop}
\begin{proof}
Let $T$ be an elementary abelian subgroup of $EV_+$ of order $\geq q^2p$. Assume towards a contradiction that $T\not\leq V_+$.
It follows from $|V_+\cap T|\geq qp$ that $|{\bf C}_{V_+}(t) |\geq qp$ for each $t\in T$. Take $t\in T\setminus V_+$. Then $t = e v_+$ for some $e\in E\setminus\{1\}$ and $v\in V$ implying
${\bf C}_{V_+}(t)={\bf C}_{V_+}(e)$. But any non-identical element of $E$ centralizes $q$ elements of ${V_+}$. A contradiction.\qed
\end{proof}
\begin{prop}\label{centralizer} If $g\in FEV_+\setminus V_+$, then the index of $\bC_{V_+}(g)$ in $V_+$ is at least $p^2$.
\end{prop}
\begin{proof} Let $g = fev_+$ be an element from $FEV_+\setminus V_+$ where $f\in F,e\in E,v\in V$. Then $\bC_{V_+}(g) = \bC_{V_+}(fe) = (\fix_V(fe))_+$. Thus we have to prove that $[V:\fix_V(fe)]\geq p^2$.

 If $f=1$, then $\fix_V(e) = W$ where $W:= \{(x,0,0,)^t\,|\,x\in \F_q\}$ and we are done

If $f\neq 1$, then the intersection $\fix_V(fe)\cap W = \fix_W(f)$ consists of those vectors $(x,0,0,)^t$ which satisfy $f(x) = x$. Hence
$$[V:\fix_V(fe)] \geq [W:\fix_W(fe)] = [W:\fix_W(f)] = [\F_q:\fix_{\F_q}(f)] > p^2.$$
\qed
\end{proof}
\begin{prop}\label{E} $\bE(QEV_+)=V_+,\bE(QE)=E$ and $\bE_2(QEV_+)=EV_+$
\end{prop}
\begin{proof}
Let $Y$ be an elementary abelian subgroup of $QEV_+$ of maximal order. Then $|Y|\geq q^3$.
We are going to prove that $Y=V_+$.

 Since $QEV_+/(EV_+)\cong Q$ is cyclic and $Y$ is elementary abelian, the image of $Y$ in the factor-group $QEV_+/(EV_+)$ has order at most $p$. Therefore $|Y\cap EV_+|\geq |Y|/p \geq q^3/p$.
If $q > p$, then $|Y\cap EV_+|\geq q^2p$. If $q=p$, then $Q$ is trivial and $|Y\cap EV_+| = |Y|\geq q^3$.
Thus in any case $|Y\cap EV_+| \geq q^2p$.
By Proposition~\ref{regsubgroups} $Y\cap EV_+\leq V_+$ implying that
$|Y\cap V_+| = |Y\cap EV_+|\geq |Y|/p\geq q^3/p$. If $Y$ is contained in $V_+$, then $Y=V_+$ and we are done. If $Y\not\leq V_+$, then an element $y\in Y\setminus V_+$ centralizes at least $q^3/p$ elements of $V_+$, contrary to Proposition~\ref{centralizer}. Thus $Y = V_+$, and, consequently, $\bE(QEV_+) = V_+$

Consider now the factor-group $QEV_+/V_+\cong QE$. The group $E$ is isomorphic to $\F_q$, while $Q$ is a cyclic $p$-group acting on $E$ as a group of field automorphisms. In this case $E$ is the only elementary abelian $p$-subgroup of maximal order. Therefore $\bE(QE)=E$ and $\bE_2(QEV_+)=EV_+$.\qed
\end{proof}\medskip

\noindent {\bf Proof of Theorem~\ref{sylow}.}\\
Let $P$ be a Sylow $p$-subgroup of $G$ containing $QEV_+$
and $N:={\bf N}_P(QEV_+)$. By Proposition~\ref{E} the subgroups $V_+$ and $EV_+$ are characteristic in $QEV_+$. Therefore both $V_+$ and $EV_+$ are normal in $N$. This implies that
 $N\leq{\bf N}_{\sym(V)}(V_+) = \aut(V) V_+$ where $\aut(V)\cong GL_{3n}(p)$. Since $EV_+\normaleq N$, the point stabilizer $N_0$ satisfies the following inequality $E\normaleq N_0 \leq \aut(V)$. By Proposition~\ref{norm} $N_0\leq FKLE$. Since $N_0$ is a $p$-group, it is contained in a Sylow's $p$-subgroup of $FKLE$. W.l.o.g. we may assume that $N_0\leq F_p LE$. It follows from $E\leq N_0\leq (F_pL)E$ that $N_0 = (N_0\cap F_p L)E$. In order to prove Theorem~\ref{sylow} it is sufficient to show that $N_0\cap F_p L = Q$. Notice that the inclusion $Q\leq N_0$ follows from the inclusion $QEV_+\leq N$.

The rest of the proof is given in the following two statements.

\begin{prop}\label{fixes} $N_0$ fixes $D$ setwise.
\end{prop}

\begin{proof} First, notice that we can assume that $O_\bullet\subseteq D$ (otherwise we can replace $D$ by $-D$).

The subgroup $N_0$ permutes the blocks of $\cD$ containing $0$. The number of such blocks is $\frac{q - 1}{2}$ - coprime to $p$. Since $N_0$ is a $p$-group, it fixes at least one of these blocks, say $D+w$. It follows from $E\leq N_0$ that $D+w$ is also fixed by $E$. Since $E$ fixes $0$, $D$ and $D+w$, it fixes also $w$ (Proposition~\ref{pointblock}, part (2)). Therefore $w\in \fix_V(E)=W:= \{(x,0,0)^t\,|\,x\in \F_q\}$,

Since $N$ normalizes $EV_+$, the center ${\bf Z}(EV_+) = \bC_{E}(V_+) = W_+$ is normal in $N$. Therefore the orbits of $W_+$ (which are the cosets of $W$) form an imprimitivity system of $N$. Hence the block  $W$ is $N_0$-invariant and $N_0$ fixes $(D+w)\cap W = (D\cap W)+w = O_{\bullet}+w$ setwise. The restriction $N_0^W$ is contained in the automorphism group of the subdesign $\cD_W:=\{D'\cap W\,|\,D'\in\cD\}$ (which is a Paley design over $\F_q$). Since $N_0\leq FEL$ and $FEL O_\bullet = O_\bullet$, the subgroup $N_0$ stabilizes $O_\bullet$ setwise. Since $Q_\bullet\subset D$, both $O_\bullet$ and $O_\bullet + w$ are blocks of the design $\cD_W$ fixed by $N_0$. By Proposition~\ref{pointblock}, part (2) applied to $\cD_W$ we conclude that $w^{N_0}=w$. Now Proposition~\ref{pointblock}, part (1) yields us $D^{N_0}=D$.\qed
\end{proof}

\begin{prop}\label{function_veps}
$Q = N_0\cap (F_p L)$.
\end{prop}
\begin{proof} Since $F$ normalizes $L$, we can replace $N_0\cap (F_p L)$ by $N_0\cap (L F_p)$.
Let $\hat{\ell}f,\ell\in{\sf End}(\F_q),f\in{\sf Gal}(\F_q/\F_p)$ be an element of the intersection $N_0\cap (L F_p)$.
By Proposition~\ref{fixes} $\hat{\ell}fD = D$. Pick arbitrary $i\in\F_q$ and $s\in\F_q^{*2}$.
By definition of $D$ the vector $v:=\veps(i)\clmn{s\left(\frac{x^2}{2}+i\right)}{sx}{s}$ is contained in $D$.
Therefore
$$
\hat{\ell}f v\in D\iff
\veps(i)\clmn{f(s)\left(\frac{(f(x))^2}{2}+f(i)\right)+\ell(f(s))}{f(s)f(x)}{f(s)}\in D
$$
Since $\clmn{f(s)\left(\frac{(f(x))^2}{2}+f(i)\right)+\ell(f(s))}{f(s)f(x)}{f(s)}\in  O_{i'}$
 where $i' ={f(i) + f(s)^{-1}\ell(f(s))}$, we conclude that $D\cap \veps(i)O_{i'}\neq\emptyset$ .But $D\cap (O_{i'}\cup -O_{i'}) = \veps(i') O_{i'}$. Therefore
$\veps(i') = \veps(i)$ implying that
 $\veps(i)=\veps(f(i) + f(s)^{-1}\ell(f(s)))$ holds for all $i\in \F_q$ and $s\in\F_q^{*2}$. This implies that $f(J) + s^{-1}\ell(s) = J$ holds for each $s\in\F_q^{*2}$. If $s^{-1}\ell(s)\neq t^{-1}\ell(t)$ for some $t,s\in\F_q^{*2}$, then $J + u = J$ where $u = s^{-1}\ell(s) - t^{-1}\ell(t)$. In this case $J$ would be a union of $\sg{u}$-cosets which is impossible because $|J|=(q\pm 1)/2$. Therefore $s^{-1}\ell(s), s\in\F_q^{*2}$ is constant, or, in other words, $\ell(s) = \alpha s, s\in\F_q^{*2}$
for some $\alpha\in\F_q$. Since $\ell$ is $\F_p$-linear, we conclude that $\ell(x)=\alpha x, x\in \F_q$, or, equivalently, $\ell = \e{\alpha}$. Thus $\hat{\ell}\in U$ and $\hat{\ell}f\in U F_p$. Since $f(J) + \alpha = J$, Proposition~\ref{Q} implies that  $\hat{\ell}f$ stabilizes $D$ setwise, that is $\hat{\ell}f\in Q$.
\qed
\end{proof}

\medskip

\subsection{ Proof of Theorem~\ref{main2}}
Let $D_\veps,D_\delta\in\DD$ be two difference sets for which the designs $\cD_\veps: = \des(V,D_\veps),\cD_\delta:=\des(V,D_\delta)$ are isomorphic. Then there exists a permutation $g\in\sym(V)$ such that
$\cD_\delta^g = \cD_\veps$. Clearly that $G_\delta^g = G_\veps$ where $G_\delta,G_\veps$ are the automorphism groups of the corresponding designs.
This implies that $(Q_\delta E V_+)^g$ is a Sylow $p$-subgroup of $G_\veps$. Since $Q_\veps EV_+$ is also a Sylow $p$-subgroup of $G_\veps$, there exists $h\in G_\veps$ such that
$(Q_\delta E V_+)^{gh} =Q_\veps E V_+$. Hence $\bE_2((Q_\delta E V_+)^{gh}) =\bE_2(Q_\veps E V_+)$.
It follows from the definition of $\bE_2$ and $\bE$ that $\bE((Q_\delta E V_+)^{gh}) = \bE(Q_\delta E V_+)^{gh}$ and $\bE_2((Q_\delta E V_+)^{gh}) = \bE_2(Q_\delta E V_+)^{gh}$.
Applying now Proposition~\ref{E} we obtain that $V_+^{gh}=V_+$ and  $(EV_+)^{gh}=EV_+$.
 It follows from $\fix(E)\neq\emptyset$ that $\fix(E^{gh})\neq \emptyset$. Therefore $E^{gh}\leq (EV_+)_v$ for some $v\in V$. Together with $|(EV_+)_0|=|E|=q^3$ we obtain that
$E^{gh}$ is a point stabilizer of $EV_+$. Therefore there exists $b\in EV_+$ such that $E^{ghb} = E$. Denoting
$g':=ghb$ we obtain that $E^{g'}=E,V_+^{g'}=V_+$, $\cD_\delta^{g'}=\cD_\veps^{hb}=\cD_\veps$ and $G_\delta^{g'}=G_\veps$. Thus $g'\in\bN_{\sym(V)}(V_+)\cap \bN_{\sym(V)}(E) = \bN_{\aut(V)V_+}(E)$.
The factor-group $\bN_{\aut(V)V_+}(E)V_+/V_+$ is embedded into $\bN_{\aut(V)}(E) = FKLE$. Together with $FKLE\leq \bN_{\aut(V)V_+}(E)$ we obtain that $\bN_{\aut(V)V_+}(E)=FKLE(\bN_{\aut(V)V_+}(E)\cap V_+) =
FKLE\bN_{V_+}(E)$. A direct computation shows that $\bN_{V_+}(E) = W_+$ where $W=\{(x,0,0)^t\,|\, x\in \F_q\}$. Thus $g'\in FKLEW_+$.

We claim that $S^{g'}$ and $S$ are conjugate by an element of $EW_+$. First we notice that $S^{g'}\leq FKLEW_+$ because $S\leq FKLEW_+$. The subgroup $S$ is a cyclic subgroup of $\bN_{\aut(V)V_+}(E)$ of order $\frac{q-1}{2}$ which centralizes $E$. Since $E^{g'}=E$, $(V_+)^{g'}=V_+$ and $[S,E]=1$, the subgroup $S^{g'}$  centralizes $E$ too. Thus the subgroup $T:=\sg{S,S^{g'}}$ is contained in $\bC_{FKLEW_+}(E)$.
Since $LEW_+$ centralizes $E$, we can write $\bC_{FKLEW_+}(E) = \bC_{FK}(E)LEW_+$.
A direct computation shows that $\bC_{FK}(E) = K_0$ where $K_0:=\{K(\alpha,1)\,|\,\alpha\in\F_q^*\}$.
Therefore $T\leq K_0 LEW_+$. Since $[L,E]=[E,W_+]=[L,W_+]=1$, the subgroup $LEW_+$ is  elementary abelian. The images of $S$ and $S^{g'}$ in the factor-group $K_0 LEW_+/(LEW_+)\cong K_0\cong\F_q^*$ coincide. Therefore $T = S(T\cap LEW_+)$. Since $LEW_+$-orbit of $0$ coincides with $W = W_+ 0$, we can write that $T\cap LEW\leq (T\cap LEW_+)_0 W_+$. Since $LEW_+$ centralizes $E$ and normalizes $V_+$, the subgroup
$(T\cap LEW_+)_0$
normalizes $EV_+$. Therefore $(T\cap LEW_+)_0 EV_+$
 is a $p$-subgroup of $G_\veps$.
Hence it is contained in the Sylow's $p$-subgroup $P$ of $G_\veps$.
 It follows from Theorem~\ref{sylow} and Proposition~\ref{E} that the order of a maximal elementary abelian subgroup of the point stabilizer $P_0$ is equal to $q$. The subgroup $(T\cap LEW_+)_0 E$ is contained in $P_0$ and is elementary abelian of order at least $q = |E|$. Therefore $(T\cap LEW_+)_0 E = E$ implying that $(T\cap LEW_+)_0\leq E$, and, consequently, $T\cap LEW_+\leq EW_+$.

Finally $T\leq SEW_+$ implying $S^{g'}\leq SEW_+$. Both $S$ and $S^{g'}$ are Hall $p'$-subgroups of
$SEW_+$. Therefore they are conjugate in $SEW_+$, say by an element $t$. Since $t\in SEW_+$, it is an automorphism of $\cD_\veps$ implying that $\cD_\delta^{g't} = \cD_\veps$. But the element $g't$ normalizes $E$,$S$ and $V_+$. Therefore $g't\in\bN_{\aut(V)V_+}(SE)$. Since $g't$ normalizes $SE$, it fixes $\fix_V(SE)$
setwise. But $\fix_V(SE)$ contains a unique point, namely $0$. Hence $g't$ fixes $0$ implying $g't\in \bN_{\aut(V)}(SE)$.  \qed

\section{Concluding Remarks.}

The first time I heard about Feng's result
was this August when I met Qing Xiang at Rogla's Conference (Slovenia).
Next month Bill Kantor communicated me about
this result independently, he also conjectured that the construction
proposed by Feng should also work for non-abelian groups of order $q^3$,
where $q$ is $3$ modulo $4$. Finally provided information
as well as fruitful discussions
with Bill during his stay in Israel served for me as the source of
inspiration to work on this project.
I am very grateful to Bill Kantor for his attention
and fruitful ideas expressed by him.
 The author also thanks Misha Klin for helpful phone conversations related to this project. My special thanks to  Matan Ziv-Av who made computer computations
(using GAP) for the group $\Z_7^3$ and thus confirmed and
clarified experimentally the initial idea,
which was generalized later for arbitrary $q$.

It seems that a modification of the construction given here could also work
in the case when $q$ congruent $1$ modulo $4$. In this case one should obtain
exponentially many partial difference sets with Paley parameters.

\end{document}